\documentclass[11pt]{article}

	\usepackage{mathtools}
	\usepackage{amssymb}
	\usepackage{amsthm}
	\usepackage{amsmath}
	\usepackage{mathrsfs}
	\usepackage{tikz}
	\usepackage{tikz-cd}
		\usetikzlibrary{decorations.pathmorphing}
	\usepackage{xcolor} 
	\usepackage{stmaryrd}
	\usepackage{bbm}
	\usepackage{adjustbox}	
	\usepackage{setspace}
	\usepackage{geometry}
	\usepackage{array} 
		\newlength\mylen
		\newcolumntype{C}{>{\hfil$}p{\mylen}<{$\hfil}}
	\usepackage{graphicx}
		\graphicspath{ {./Images/} }
	\usepackage{caption}
	\usepackage{subcaption}

\DeclareMathOperator{\Perv}{Perv}
\DeclareMathOperator{\mix}{mix}

\DeclareMathOperator{\Parity}{Parity}

\DeclareMathOperator{\C}{\mathbb{C}}

\DeclareMathOperator{\Z}{\mathbb{Z}}

\DeclareMathOperator{\bH}{\mathbb{H}}

\DeclareMathOperator{\sE}{\mathscr{E}}

\DeclareMathOperator{\tA}{\textbf{A}}

\DeclareMathOperator{\tR}{\textbf{R}}

\DeclareMathOperator{\Hom}{Hom}
\DeclareMathOperator{\End}{End}
\DeclareMathOperator{\rank}{rk}
\DeclareMathOperator{\grrk}{rk^{\bullet}}
\DeclareMathOperator{\id}{id}

\DeclareMathOperator{\Sym}{Sym}

\DeclareMathOperator{\Tr}{\mathrm{Tr}}
\DeclareMathOperator{\CH}{\mathrm{CH}}
\DeclareMathOperator{\hH}{\mathrm{H}}

\newtheorem{thm}{Theorem}
\newtheorem{cor}[thm]{Corollary}
\newtheorem{lem}[thm]{Lemma}
\newtheorem{prop}[thm]{Proposition}

\newtheorem*{thm*}{Theorem}
\theoremstyle{remark}
\newtheorem{rem}[thm]{Remark}

\theoremstyle{definition}
\newtheorem{thmalt}{Theorem}[thm]
\newenvironment{thmp}[1]{
  \renewcommand\thethmalt{#1}
  \thmalt
}{\endthmalt}

\begin{document}

\title{The geometry of tilting composition series via Richardson varieties}
\author{Joseph Baine and Chris Hone}
\date{}

\newcommand{\Addresses}{{
  \bigskip
  \footnotesize

  J.~Baine, \textsc{Max-Planck-Institut f\"{u}r Mathematik, Vivatsgasse 7, 53111 Bonn, Germany}\par\nopagebreak
  \textit{E-mail address}: \texttt{baine@mpim-bonn.mpg.de}

  \medskip

  C.~Hone, \textsc{Department of Mathematical Sciences, University of Copenhagen, Denmark}\par\nopagebreak
  \textit{E-mail address}: \texttt{cth@math.ku.dk}
}}

	\maketitle
	\begin{abstract}
		\noindent
		We prove the (graded) Jordan--H\"{o}lder multiplicities of (mixed) tilting sheaves on flag varieties admit a geometric interpretation as the hypercohomology of certain sheaves on Richardson varieties in the Langlands dual flag variety. These sheaves are a motivic variant of geometric extensions, and may be described as a tensor product of parity sheaves on the Schubert and opposite Schubert varieties. We also provide an explicit formula for these multiplicities in terms of $\ell$-Kazhdan--Lusztig polynomials.
	\end{abstract}
	
\section{Introduction}

    Tilting objects in highest weight categories have emerged as the salient objects of interest in geometric representation theory. 
	Understanding their structure is one of the major open problems in the field. 
	For example, determining the structure of tilting sheaves on finite and affine flag varieties would yield: the characters of simple and tilting modules for algebraic groups in positive characteristic \cite{AMRW19, RW21}; decomposition numbers for representations of the symmetric group \cite{Erd94}; characters for simple representations of finite groups of Lie type in equal characteristic \cite{Ste63}; and tensor product formulae for certain semisimple symmetric tensor categories \cite{GK92}.
	\\
	\par 
	Often it is difficult to study tilting objects directly. 
	When studying tilting sheaves on flag varieties, a particularly fertile approach has been to exploit deep equivalences, like modular Koszul duality, to reinterpret structural properties of tilting sheaves in terms of the stalks of parity sheaves on the Langlands dual flag variety \cite{AMRW19}. 
    However, a description of the image of intersection cohomology sheaves under modular Koszul duality is lacking, and until now  a geometric interpretation of the Jordan--H\"{o}lder multiplicities for tilting sheaves was unknown.
	\\
    \par 
	Fix a complex reductive group $G$ and field $\Bbbk$ of characteristic $\ell\geq 0$. Given an indecomposable tilting sheaf $\overline{T}_x$ with coefficients in $\Bbbk$, supported on a Schubert variety $X_x \subseteq G/B$, and the simple intersection cohomology sheaf $\overline{L}_z$ supported on $X_z \subseteq G/B$, one can ask for a natural geometric interpretation of the Jordan--H\"{o}lder multiplicity $[\overline{T}_x : \overline{L}_z]$. 
	The main result of this paper is such an interpretation as the rank of the hypercohomology of a certain sheaf  $\check{\sE_x^z}$ supported on the Richardson variety $\check{X}_x^z$ in the Langlands dual flag variety $\check{G}/\check{B}$. 
	More precisely, we prove:
	
	\begin{thm}
\label{Thm: Ungraded multiplicities}
		Let $\text{char } \Bbbk = \ell$ be either 0 or a good prime for $G$, then 
		\begin{align*}
			[\overline{T}_x: \overline{L}_z] 
			= 
			\rank\bH^{\bullet}(\check{\sE_x^z}). 
		\end{align*}
	\end{thm}
	
	This sheaf $\check{\sE_x ^z}$ is a motivic variant of the geometric extension recently defined by Williamson and the second author in \cite{HW23} and independently by McNamara in \cite{McN18}. 
    Even in characteristic 0, where the geometric extension $\check{\sE_x^z}$ is the intersection cohomology sheaf supported on $\check{X}_x^z$, this result  is new. 
	\\
	\par
	This result is deduced from a stronger statement in the context of mixed tilting sheaves. 
	Since the seminal work of Beilinson, Ginzburg, and Soergel in \cite{BGS96}, it has been known that\ equivalences like Koszul duality are only visible when working with $\Z$-graded categories.
	In characteristic zero, the $\mathbb{Z}$-graded versions of these categories were constructed using various theories of weights \cite[\S4]{BGS96}. 
    Recently Achar and Riche have constructed $\Z$-graded categories of mixed sheaves with coefficients in a field $\Bbbk$ of characteristic $\ell \geq 0$ \cite{AR16}.
	These categories have analogous indecomposable mixed tilting sheaves $T_x$ and simple mixed intersection cohomology sheaves $L_x \langle i \rangle$, where $i \in \Z$ . 
	The stronger and mixed version of Theorem \ref{Thm: Ungraded multiplicities} is the following: 
	
	\begin{thm}\label{thm:main thm}
		Let $\text{char } \Bbbk = \ell\geq 0$. For any $\ell$, we have 
		\begin{align*}
			[T_x: L_z \langle i \rangle] 
			= 
			\rank \bH^{i}(\check{\sE_x^z}). 
		\end{align*}
	\end{thm}

	\par
    Further, we determine an explicit formula for the graded rank of the hypercohomology of the geometric extension $\sE_x^z$ on the Richardson variety $X_x^z$ in a flag variety $G/B$. 
	This generalises the main result of Dyer and Lusztig in \cite{DL23} from characteristic 0 to positive characteristic.	
	\begin{thm}\label{thm:geometric ranks}
		For any characteristic $\ell \geq 0$, the graded rank of $\bH^{\bullet}(\sE_x^z)$ is given by:
		\begin{align*}
			\grrk \bH^{\bullet}(\sE_x^z) 
			= 
			\sum_{z \leq y \leq y' \leq x} {}^{\ell} h_{y',x}  ~ r_{y,y'} ~ {}^{\ell} h_{w_0 y, w_0 z}
			= 
			\sum_{z \leq y \leq x} {}^{\ell} h_{y,x} ~\overline{{}^{\ell} h_{w_0 y, w_0 z}} ,
		\end{align*} 
        where ${}^{\ell} h_{y, x}$ is an $\ell$-Kazhdan--Lusztig polynomial, and $r_{y,y'}$ is a renormalisation of the Kazhdan--Lusztig $R$-polynomial. In particular, the cohomology $\bH^{\bullet}(\sE_x^z)$ is supported in a single parity.
	\end{thm}

    Using modular Koszul duality we may categorify the equality of Theorem \ref{thm:main thm} for the special case of Schubert varieties in Section \ref{Ssec: Isomorphisms}. We also provide an explicit sheaf theoretic description of the geometric extension, providing a geometric interpretation of the numerical equality in Theorem \ref{thm:geometric ranks}.
    \par 
    \begin{thm}
\label{Thm: Tensor product}
    The geometric extension on the Richardson variety is the tensor product \[\mathscr{E}^z_x\cong \mathscr{E}_x\otimes \mathscr{E}^z\] where $\mathscr{E}_x$ and $\mathscr{E}^z$ denote the dense indecomposable parity sheaves on the Schubert and opposite Schubert varieties respectively.
    \end{thm}
    
    We prove Theorem \ref{thm:main thm} by computing this multiplicity in two ways. In Section \ref{Sec: JH mult} we exploit various equivalences associated to highest weight categories to determine an explicit formula for $[T_x: L_z \langle i \rangle]$ as a sum of $\ell$-Kazhdan--Lusztig polynomials. 
	The terms in this sum are naturally indexed by strata in the Richardson variety $\check{X}^z_x$, and the $\ell$-Kazhdan--Lusztig polynomials occurring are exactly those predicted by the local model of the Richardson singularities given by Knutson, Woo, and Yong \cite{KWY13}. 
	This suggests interpreting both sides as the (polynomial) trace of Frobenius on a mixed $\ell$-adic sheaf on $\check{X}^z_x$ which has the correct stalks locally. We prove Theorem \ref{thm:main thm} via this strategy in Section \ref{Subsec: geometric argument}, assuming the existence of the sheaves $\check{\mathscr{E}}^z_x$.
    In Section \ref{subsection:geometric extensions on rich} we construct the required sheaves, which are a motivic variant of the geometric extensions introduced in \cite{HW23,McN18}, and prove their alternate description as a tensor product in Section \ref{Ssec: Tensor product}.
    
    \par 
	\emph{Acknowledgments:}
    The first author was partially supported by ARC grant DP220102861. 
    He thanks the University of New South Wales, and the Max Planck Institute of Mathematics in Bonn for their hospitality. 
	The second author was supported by the Danish National Research Foundation through the Copenhagen Centre for Geometry and Topology (DNRF151).
    The authors thank Sam Jeralds for bringing \cite{KLS14} to our attention.

\section{Jordan--H\"{o}lder multiplicities of tilting sheaves}
\label{Sec: JH mult}
	In this section we determine a formula for the (graded) Jordan--H\"{o}lder multiplicities of (mixed) tilting sheaves on flag varieties in terms of $\ell$-Kazhdan--Lusztig polynomials. 
	The formula follows from various dualities associated to highest weight and Hecke categories. 

\subsection{Recollections on mixed tilting sheaves}
\label{Subsec: tilting sheaves}

	Throughout this section $G$ denotes a complex reductive algebraic group, with fixed Borel subgroup $B$, fixed maximal torus $T \subset B$, and Weyl group $W$. 
	The choice of Borel endows $W$ with a distinguished set $S$ of generators, such that $(W,S)$ is a Coxeter system with length function $\vert \, \cdot \, \vert : W \rightarrow \Z_{\geq 0}$. 
	The flag variety $G/B$ admits a Schubert stratification by $B$-orbits:
	\begin{align*}
		G/B = \bigsqcup_{x \in W} B xB/B.
	\end{align*}
	Each Schubert cell $B xB/B$ is isomorphic to the affine space $\C^{|x|}$, and the Schubert variety $X_x$ is the closure of the Schubert cell $ B xB/B$. 
	For later use, we define the opposite Borel subgroup $B^-$ as the unique Borel subgroup satisfying $B \cap B^- = T$; the opposite Schubert variety $X^x$ is the closure of the $B^-$-orbit $B^-  x B/B$. 
	\\
	\par 
	Fix a field\footnote{
    More generally, $\Bbbk$ can be a complete local ring whose residue field has characteristic $\ell \geq 0$. 
    In this case $D_{(B)}^b (G/B , \Bbbk)$ and $\overline{\Perv}(G/B ,\Bbbk)$ are defined, although the latter is not a highest weight category. 
    One still has sheaves $\overline{L}_x$, $\overline{\Delta}_x$, $\overline{\nabla}_x$, and $\overline{T}_x$, see \cite{Jut09,AR16a}, and the stalks we consider agree with their modular reductions.} $\Bbbk$ of characteristic $\ell \geq 0$.
	Denote by $D_{(B)}^b (G/B , \Bbbk)$ the bounded derived category of sheaves of $\Bbbk$-modules on $G/B$ that are constructible with respect to the Schubert stratification.  Let $(1)$ denote its (triangulated) shift functor\footnote{This notation is standard when considering mixed Hecke categories as the notation $[1]$ is reserved for the triangulated shift functor on the bounded homotopy category of parity sheaves.}.
	This category has a perverse $t$-structure, whose heart we denote by $\overline{\Perv}(G/B , \Bbbk)$. 
	The category $\overline{\Perv}(G/B ,\Bbbk)$ is a finite highest weight category, so has simple $L_x$, standard $\Delta_x$, and costandard $\nabla_x$ objects indexed by elements of $W$. In particular, if $\mathbf{1}_{BxB/B}$ denotes the constant sheaf on $BxB/B$ and $i_x: BxB/B \hookrightarrow G/B$ the inclusion, then for each $x \in W$, these distinguished objects are given by the following perverse sheaves:  
	\begin{align*}
		\overline{\Delta}_x := i_{x!} \mathbf{1}_{BxB/B} (|x|), 
        &&
        \overline{L}_x := i_{x!*} \mathbf{1}_{BxB/B} (|x|)=\text{IC}(X_x,\Bbbk),
		&&
		\overline{\nabla}_x := i_{x*} \mathbf{1}_{BxB/B} (|x|).
	\end{align*}
    In any highest weight category, we have another distinguished class of objects indexed by elements of the underlying poset, the (indecomposable) tilting objects. 
    Recall that an object in a highest weight category is called tilting if it admits a $\Delta$-filtration (a filtration where successive subquotients are isomorphic to standard objects), and a  $\nabla$-filtration (defined analogously). 
    The axioms of a highest weight category guarantee the existence of enough tilting objects; in our setting, there is a unique indecomposable tilting sheaf $\overline{T}_x$ supported on $X_x$; and these exhaust all isomorphism classes of indecomposable tilting sheaves. 
    \\
    \par 
	In sharp contrast to the simple, standard, and costandard sheaves, there is not any simple geometric construction of tilting sheaves on the flag variety.
	To understand tilting sheaves geometrically, one may pass through modular Koszul duality of \cite{AR16, AMRW19}, and use parity sheaves on the Langlands dual flag variety.
    \\
	\par 
	A $B$-constructible parity sheaf $\mathcal{E}$ on $G/B$, as in \cite{JMW14}, is a direct sum of sheaves $\mathcal{F}$ satisfying: both $\mathcal{H}^n(i_x^!\mathcal{F})$ and $\mathcal{H}^n(i_x^*\mathcal{F})$ are zero for all $n \in \Z$ of some parity (either odd or even), and isomorphic to a free $\Bbbk$-local system otherwise.
	The indecomposable parity sheaves $\mathscr{E}_x(i)$ are indexed by $x\in W$ and $i \in \Z$, as
    shown in \cite[\S4.1]{JMW14}. 
    The parity sheaf $\mathscr{E}_x$ is supported in $X_x$, and is Verdier self-dual. 
     The full subcategory of parity sheaves is denoted $\Parity(G/B, \Bbbk)$.
	Given a parity sheaf $\mathscr{E}$ we denote its stalk at $yB/B$ by $i_y^*\mathscr{E}$.
	The stalks of parity sheaves are encoded by $\ell$-Kazhdan--Lusztig polynomials ${}^{\ell} h_{y,x}$.
	In particular, we define
	\begin{align}
\label{Eqn: lKL polynomials}
		{}^{\ell} h_{y,x}(v) 
		:=
		\sum_{i \in \Z}  \rank_{\Bbbk} H^{i} (i_y^*\mathscr{E}_{x}) v^{-|y| - i},
	\end{align}
	where $\rank_{\Bbbk}$ denotes the rank as a $\Bbbk$-module. 
    \begin{rem}
     Parity sheaves exist only in restricted situations, but on Schubert varieties $X_x$, the indecomposable parity sheaf is isomorphic to the geometric extension of \cite{HW23}. Later we consider Richardson varieties, where the formalism of parity sheaves is not applicable. 
    In that setting, geometric extensions provide the necessary replacement.
    \end{rem}

	\par 
	Following Achar and Riche \cite{AR16}, we consider $D^{\mix}(G/B, \Bbbk) := K^b(\Parity(G/B, \Bbbk))$ the bounded homotopy category of parity sheaves.
	Denote its homological shift functor by $[1]$. This category admits a canonical perverse $t$-structure, whose heart we denote by $\Perv(G/B, \Bbbk)$. 
	The category $\Perv(G/B, \Bbbk)$ is a graded highest weight category with grading shift functor $\langle 1 \rangle := [1](-1)$.	The isomorphism classes of simple $L_x \langle i \rangle$, standard $\Delta_x \langle i \rangle$, costandard $\nabla_x \langle i \rangle$, and indecomposable tilting $T_x \langle i \rangle$ sheaves are parameterised by $x \in W$ and $i \in \Z$. The category $\Perv(G/B, \Bbbk)$ is a $\Z$-graded enhancement of $\overline{\Perv}(G/B,\Bbbk)$; this shall be discussed further in section \ref{subsec: ungraded JH}.
	\\
	\par 
	Given a sheaf $\mathcal{F}$ admitting a $\Delta$-filtration, we write $(\mathcal{F}: \Delta \langle i \rangle)$ for the number of subquotients isomorphic to $\Delta \langle i \rangle$ in any $\Delta$-filtration of $\mathcal{F}$; if $\mathcal{F}$ admits a $\nabla$-filtration $(\mathcal{F}: \nabla \langle i \rangle)$ is defined analogously. 
	These multiplicities are independent of the choice of filtration.
	Since $\Perv(G/B, \Bbbk)$ is a finite highest weight category, it has enough projectives, and any projective object has a $\Delta$-filtration\footnote{
    Although $\Perv(G/B, \Bbbk)$ is not a highest weight category when $\Bbbk$ is a complete local ring, the existence and properties of indecomposable tilting sheaves and projective covers still hold, see \cite[\S3.3]{AR16}. 
    } \cite[\S A]{AR16}. 
	The projective cover of $L_x$ is denoted $P_x$.  
	Graded BGG reciprocity states
	\begin{align}
\label{Eqn: BGG}
		[\nabla_y : L_z \langle i \rangle] = (P_z \langle i \rangle : \Delta_y).
	\end{align}
	Moreover, the category $D^{\mix}(G/B, \Bbbk)$ admits a covariant auto-equivalence $\tR$; called Ringel duality, which takes costandard objects to standard objects and tilting objects to projective objects \cite[\S4]{AR16}. 
	More precisely, if $w_0$ denotes the longest element of $W$ then $\tR(\nabla_{x}) \cong \Delta_{w_0 x}$ and $\tR(T_{x}) \cong P_{w_0 x}$. 
	As a consequence we obtain another description of this multiplicity:
	\begin{align}
\label{Eqn: Ringel}
		(P_z \langle i \rangle : \Delta_y) = (T_{w_0 z} \langle i \rangle : \nabla_{w_0 y}).
	\end{align}

    The multiplicities $(T_x : \nabla_y \langle j \rangle)$ have a geometric interpretation. 
	Let $\check{G}$ denote the complex reductive group that is Langlands dual to $G$,
	and $\check{\mathscr{E}}_x$ and $\check{\nabla}_x$ the corresponding parity and costandard sheaves on the Langlands dual flag variety $\check{G}/ \check{B}$. 
    Modular Koszul duality of \cite{AR16,AMRW19,RV24} is a triangulated equivalence 
    \begin{align*}
    	\kappa: D^{\mix}(G/B, \Bbbk) \tilde{\longrightarrow} D^{\mix}(\check{G}/\check{B}, \Bbbk)
    \end{align*}
    satisfying, among other things, $\kappa(T_x \langle i\rangle ) \cong \check{\mathscr{E}}_x(i)$ and $\kappa(\nabla_x \langle i\rangle ) \cong \check{\nabla}_x(i) $. 
    If we define ${}^{\ell}\check{h}_{y,x}$ for $\check{G}$, as in Equation (\ref{Eqn: lKL polynomials}), then modular Koszul duality implies 
    \begin{align}
\label{Eqn: Koszul}
		{}^{\ell} \check{h}_{y,x}
		=
		\sum_{i \in \Z} (T_x : \nabla_y \langle i \rangle) v^i.
	\end{align}

\subsection{Graded Jordan--H\"{o}lder multiplicities}
	
	We now determine a formula for the graded Jordan--H\"{o}lder multiplicities of tilting sheaves in terms of $\ell$-Kazhdan--Lusztig polynomials for the Langlands dual group.  
	
	\begin{prop}
\label{Prop: Tilting JH mult}
		The following identity holds
		\begin{align*}
			\sum_{i \in \Z} ~[T_x : L_z \langle i \rangle] \, v^i
			=
			\sum_{z \leq y \leq x} {}^{\ell} \check{h}_{w_0 y, w_0 z}(v) ~{}^{\ell} \check{h}_{y,x}(v^{-1})
		\end{align*}
		where ${}^\ell \check{h}_{y,x}$ denotes the $\ell$-Kazhdan--Lusztig polynomial for the Langlands dual group.
	\end{prop}
	
	\begin{proof}
	By considering the Jordan--H\"{o}lder multiplicity of $L_{z}\langle i \rangle$ in each subquotient appearing in a $\nabla$-filtration of $T_x$, we find:
	\begin{align*}
		[T_x : L_z \langle i \rangle]
		&= 
		\sum_{y \in W, j \in \Z} 
		(T_x : \nabla_y \langle j \rangle)
		[\nabla_y \langle j \rangle : L_z \langle i \rangle]
		\\
		&=
		\sum_{y \in W, j \in \Z}
		(T_x : \nabla_y \langle j \rangle)
		(P_{z} \langle i \rangle  : \Delta_{y} \langle j \rangle )
		&&
		\text{Equation (\ref{Eqn: BGG})}
		\\
		&=
		\sum_{y \in W, j \in \Z}
		(T_x : \nabla_y \langle j \rangle)
		(T_{w_0 z} \langle i \rangle : \nabla_{w_0y} \langle j \rangle)
		&&
		\text{Equation (\ref{Eqn: Ringel})}.
	\end{align*}
	The claim follows immediately from Equation (\ref{Eqn: Koszul}).
	\end{proof}

	The category $\Perv(G/B,\Bbbk)$ admits a Verdier duality $\mathbb{D}$ which satisfies $\mathbb{D}(L_x) \cong L_x$, $\mathbb{D}(T_x) \cong T_x$ and $\mathbb{D} \circ \langle 1 \rangle \cong \langle -1 \rangle \circ \mathbb{D}$. 
	These properties, together with Proposition \ref{Prop: Tilting JH mult}, imply: 
	
	\begin{cor}\label{cor:numerical duality}
		The following identity holds:
		\begin{align*}
			\sum_{z \leq y \leq x} {}^{\ell} \check{h}_{w_0 y, w_0 z}(v) ~{}^{\ell} \check{h}_{y,x}(v^{-1})
			=
			\sum_{z \leq y \leq x} {}^{\ell} \check{h}_{w_0 y, w_0 z}(v^{-1}) ~{}^{\ell} \check{h}_{y,x}(v).
		\end{align*}
	\end{cor}
	This generalises \cite[\S 6(d)]{DL23} to the modular setting. 

    \begin{rem}
		It also follows from Proposition \ref{Prop: Tilting JH mult} that mixed tilting sheaves satisfy the reciprocity formula: 
        \begin{align*}
            [T_x : L_z \langle i \rangle] = [T_{w_0 z} : L_{w_0 x} \langle -i \rangle].
        \end{align*}
	\end{rem}

\subsection{Ungraded Jordan--H\"{o}lder multiplicities}
\label{subsec: ungraded JH}

    Achar and Riche \cite[\S5.3]{AR16} show that $\Perv(G/B, \Bbbk)$ is a $\Z$-graded enhancement of $\overline{\Perv}(G/B,\Bbbk)$ in the following sense. 
	If $\ell$ is a good prime for $G$, there exists a `forget grading' functor 
	\begin{align*}
		F : D^{\mix} (G/B, \Bbbk) \rightarrow D^b_{(B)}(G/B, \Bbbk).
	\end{align*}
    The functor $F$ is $t$-exact, and on the distinguished objects in $\Perv(G/B, \Bbbk)$ it satisfies:
	\begin{align*}
		F(L_x \langle i \rangle) \cong \overline{L}_x,
		&&
		F(\Delta_x \langle i \rangle) \cong \overline{\Delta}_x,
		&&
		F(\nabla_x \langle i \rangle) \cong \overline{\nabla}_x,
		&&
		F(T_x \langle i \rangle) \cong \overline{T}_x
	\end{align*}
	for all $x \in W$ and $i \in \Z$. For any sheaves $\mathcal{F}, \mathcal{G}$ in $D^{\mix} (G/B, \Bbbk)$, $F$ induces an isomorphism  
	\begin{align*}
		\bigoplus_{i \in \Z }\Hom_{D^{\mix}(G/B, \Bbbk)}(\mathcal{F},\mathcal{G} \langle i \rangle )
		\cong 
		\Hom_{D^b_{(B)}(G/B, \Bbbk)}(F(\mathcal{F}), F(\mathcal{G})). 
	\end{align*}
	The existence of this de-grading functor $F$ implies we may set $v=1$ in Proposition \ref{Prop: Tilting JH mult} to recover the following ungraded multiplicity statement.
	
	\begin{cor}
		Let $\ell$ be a good prime for $G$. 
		The following identity holds
		\begin{align*}
			[\overline{T}_x : \overline{L}_z]
			=
			\sum_{z \leq y \leq x} {}^{\ell} \check{h}_{w_0 y, w_0 z}(1) ~{}^{\ell} \check{h}_{y,x}(1).
		\end{align*}
	\end{cor} 

\subsection{Upgrading equalities to isomorphisms}
\label{Ssec: Isomorphisms}

    Throughout Section \ref{Ssec: Isomorphisms} we assume that $p$ is a good prime for $G$, and $G$ is isomorphic to a product of $GL_n$ and of quasi-simple groups not of type $\tA$.
    This assumption allows us to identify various algebras with the coinvariant algebra, as explained in \cite[\S4.1]{AR16a}. 
    \\
    \par 
    Let $\check{R} = \Sym(X_*(T) \otimes \Bbbk)$ denote the symmetric algebra on the cocharacter lattice of $T \subset G$. 
    We define the graded homomorphism space
    \begin{align*}
        \Hom^{\bullet} (M,N) 
        := 
        \bigoplus_{i \in \Z}
        \Hom_{D^{\mix}(G/B,\Bbbk)} (M, N\langle i \rangle).
    \end{align*}
    Precomposition endows $\Hom^{\bullet} (M,N)$ with the structure of a graded right $\text{End}^{\bullet}(M)$-module. 
    By \cite[\S4.1]{AR16a} we have $\text{End}^{\bullet}(P_{\id}) \cong \check{R}/ \check{R}^{W}$, so $\Hom^{\bullet} (P_{\id},N)$ is a $\check{R}$-module.
    Since $\check{R}$ is the $\check{T}$-equivariant cohomology of a point, for any sheaf $\mathscr{F}$ in $D^{\mix}(\check{G}/\check{B}, \Bbbk)$, the hypercohomology $\bH^{\bullet}(\check{\mathscr{F}})$ has the structure of a $\check{R}$-module. 
    Theorem \ref{thm:main thm} can equivalently be stated as saying, for all $x,z \in W$, we have
    \begin{align*}
         \grrk \Hom^{\bullet} (P_z , T_x )
         =
         \grrk \bH^{\bullet}(\check{\mathscr{E}_x^z}). 
    \end{align*}
    It is natural to ask whether this can be upgraded to an isomorphism of graded $\check{R}$-modules.
    We can affirmatively answer this question when the Richardson variety is a Schubert variety. 

     \begin{prop}
\label{Prop: Categorify numerics}
        For each $x \in W$, we have isomorphisms:
        \begin{align*}
            \Hom^{\bullet}(P_{\id} , T_x) \cong \bH^{\bullet}(\check{\mathscr{E}}_x^{\id}),
            &&
            \Hom^{\bullet}(P_{x} , T_{w_0}) \cong \bH^{\bullet}(\check{\mathscr{E}}_{w_0}^{x}).
        \end{align*}
    \end{prop}

    \begin{proof}
        Let $\underline{w}$ be an expression for $w \in W$, and let  $\pi_{\underline{w}} : \check{X}_{\underline{w}} \rightarrow \check{X}_w$ be the Bott--Samelson resolution. 
        We call $\pi_{\underline{x}*}\mathbf{1}_{\check{X}_{\underline{w}}}$ a Bott--Samelson parity sheaf.
        We can also associate to $\underline{w}$ a Bott--Samelson tilting sheaf\footnote{The definition of $T_{\underline{w}}$ is technical, requiring the free monodromic tilting category of \cite{AMRW19}.} $T_{\underline{w}}$. 
        In \cite[\S 5.1]{AMRW19} it is shown that the category of Bott--Samelson parity sheaves on $\check{G}/\check{B}$ is equivalent to the category of Bott--Samelson tilting sheaves on $G/B$. 
        Further, the hypercohomology functor on Bott--Samelson parity sheaves corresponds to graded maps from the projective $P_{\id}$ on Bott--Samelson tilting sheaves \cite[\S3.6]{AMRW19}; i.e. there is an isomorphism of graded $\check{R}$-modules
        \begin{align*}
            \Hom^{\bullet} (P_{\id}, T_{\underline{w}})
            \cong
             \bH^{\bullet}(\pi_{\underline{w} *}\mathbf{1}_{\check{X}_{\underline{w}}}).
        \end{align*}
        The induced functor on the Karoubi envelope then yields an isomorphism
        \begin{align*}
            \Hom^{\bullet} (P_{\id}, T_x)
            \cong 
            \bH^{\bullet}(\check{\mathscr{E}}_x).
        \end{align*}
        The first isomorphism in the Proposition follows from the fact that $X_x^{\id}$ is the Schubert variety $X_x$. 
        The second isomorphism follows from the arguments in \cite[\S5]{Bai} and the fact that $X_{w_0}^x$ is isomorphic to the Schubert variety $X_{xw_0}$.
    \end{proof}

\subsection{An $\ell$-Kazhdan--Lusztig polynomial identity}\label{subsection:introduce r poly}
    We conclude this section by deriving an alternative version of the graded Jordan--H\"{o}lder multiplicities of tilting sheaves that is reminiscent of a point counting formula.
	\\
	\par 
	 Let $H$ denote the Hecke algebra of $(W,S)$. 
	It is the unital, associative $\Z[v,v^{-1}]$-algebra generated by the symbols $\{ \delta_x ~|~ x \in W \}$, subject to the relations:
	\begin{align*}
		(\delta_s+v)(\delta_s-v^{-1}) &=0 && \text{for each } s\in S, \text{ and}
		\\
		\delta_x \delta_y &= \delta_{xy} && \text{whenever } |x| + |y| = |xy|.
	\end{align*}
	It follows from \cite{Tits69} that $\{ \delta_x ~\vert~ x \in W\}$ is a $\Z[v,v^{-1}]$-basis of $H$.
	The algebra $H$ admits a unique $\Z$-linear involution satisfying $\overline{v} = v^{-1}$ and $\overline{\delta_x} = \delta_{x^{-1}}^{-1}$ for each $x \in W$.
	The polynomials $r_{y,x} \in \Z [v,v^{-1}]$ are defined by the equality
	\begin{align}
\label{Eqn: r pol}
		\overline{\delta_x} = \sum_y (-1)^{|xy|} ~r_{y,x}~ \delta_y. 
	\end{align}
	These are a renormalisation of the $R$-polynomials in \cite{KL79}, with $r_{y,x} = v^{|x| - |y|} R_{y,x}(v^{-2})$. 
	The only property of $r_{y,x}$ we require is
	\begin{align}
\label{Eqn: r pol bar}
		\overline{r_{y,x}} = (-1)^{|xy|} r_{y,x}
	\end{align}
	which follows immediately from their inductive construction in \cite[\S2.1(i)]{KL79}.
	\\
	\par 
	We begin with the following Lemma. 
	\begin{lem}
\label{Lem: r pol identity}
		Let $h = \sum_{x\in W} a_x \delta_x $ be any element of $H$ such that  $a_x \in \Z[v,v^{-1}]$ and $\overline{h}=h$,
		then $\overline{a_x} = \sum_{w \in W} a_w r_{x,w}$. 
	\end{lem}
	
	\begin{proof}
		Equation (\ref{Eqn: r pol}) together with the fact $\overline{h}=h$ implies $a_x = \sum_{w \in W} (-1)^{|xw|} \, \overline{a_w} \, r_{x,w}$. 
		The claim follows from applying the involution and using Equation (\ref{Eqn: r pol bar}). 
	\end{proof}
	
	The following is the main result of this section.
	
	\begin{prop}\label{prop:the polynomial formula}
    We may express the graded multiplicities of mixed simple sheaves in mixed tilting sheaves using the following formula:
		\begin{align*}
			\sum_{i \in \Z} ~[T_x : L_z \langle i \rangle] \, v^i
			=
			\sum_{z \leq y \leq x} {}^{\ell} \check{h}_{w_0 y, w_0 z}(v) ~{}^{\ell} \check{h}_{y,x}(v^{-1})=\sum_{z \leq y \leq y' \leq x} {}^{\ell} \check{h}_{w_0 y, w_0 z}(v) ~ r_{y,y'}(v) ~{}^{\ell} \check{h}_{y',x}(v).
		\end{align*}
	\end{prop}
	
	\begin{proof}
		The element ${}^\ell \check{b}_x :=  \sum_y {}^\ell \check{h}_{y,x} \, \delta_y$ satisfies $\overline{{}^\ell \check{b}_x} = {}^\ell \check{b}_x$ because the parity sheaf $\check{\mathscr{E}}_x$ is Verdier self-dual \cite[\S2.2]{JMW14}. 
		The claim then follows from Proposition \ref{Prop: Tilting JH mult} and Lemma \ref{Lem: r pol identity}.  
	\end{proof}

\section{Richardson varieties and their geometry}
\label{Sec: Geometry}

    In this section we prove that the polynomial formula of Proposition \ref{prop:the polynomial formula} also computes the Poincar\'{e} polynomial of the cohomology of a sheaf $\mathscr{E}(X^z_x)$ on the Richardson variety $X^z_x$.
    In Section \ref{Ssec: Richardson} we recall various geometric properties of Richardson varieties. We give the counting argument for Theorem \ref{thm:geometric ranks} in Section \ref{Subsec: geometric argument}, assuming the existence and formal properties of the sheaves $\mathscr{E}(X^z_x)$. 
    In Section \ref{subsection:geometric extensions on rich} we construct these sheaves as a motivic variant of the geometric extension\footnote{See the discussion preceding Proposition \ref{Prop:geo properties} for a discussion of this notational change and associated change in normalisation.} of \cite{HW23,McN18}, and in Section \ref{Ssec: Tensor product} we provide an alternate description of this sheaf as a tensor product.

\subsection{Recollections on Richardson varieties}
\label{Ssec: Richardson}
	In this section we recall the definition of Richardson varieties, the local structure of their singularities, and their resolutions of Bott--Samelson type, as determined in \cite{KWY13} and \cite{KLS14} respectively.
	 We refer the reader to \cite{Ri92,De85, KLS14, Sp24} for more background on Richardson varieties. From this point onwards we work in an algebraic context, our schemes are defined over $\mathbb{Z}$ or finite fields, while the relevant sheaf theory will be $\ell$-adic \'{e}tale cohomology. In particular, $G$, $B$, $B^-$ and $T$ denote integral forms of the complex groups of Section \ref{Subsec: tilting sheaves}. 
    \\
    \par 
	The (closed) Richardson variety $X^z_x$ associated to $z,x \in W$, with $ z \leq w$, is the intersection of the Schubert variety $X_x$ and the opposite Schubert variety $X^z$ within the flag variety $G/B$:
	\begin{align}
		X^z_x:=X_x\cap X^z=\overline{BxB/B}\cap \overline{B^{-}zB/B}.
	\end{align}
	It is an irreducible variety of dimension $|x|-|z|$. The open Richardson variety $(X^z_x)^\circ$ is the intersection of the associated Schubert cells:
	\begin{align}
		(X^z_x)^\circ:=(BxB/B) \cap (B^{-}zB/B).
	\end{align}
	It is a smooth dense open subvariety of $X^z_x$.
	Any closed Richardson variety $X^z_x$ is stratified by open Richardson varieties:
	\begin{align}
\label{Richardson_strat}
    	X^z_x=\bigsqcup_{z\leq y\leq y'\leq x}(X^y_{y'})^\circ.
	\end{align}

	As (open) Richardson varieties are defined over $\mathbb{Z}$, we may base change to $\mathbb{F}_p$ and consider their points defined over the finite fields $\mathbb{F}_q$. 
	This point count is a polynomial in $q$, i.e.
	independent of $p$. For	open Richardson varieties this polynomial is, up to normalisation, the polynomial $r_{y,y'}$  introduced in Section \ref{subsection:introduce r poly}:
	\begin{align}
		|(X^y_{y'})^\circ(\mathbb{F}_q)|
		=
		R_{y,y'}(q)
		=
		q^{\frac{|y'|-|y|}{2}}
		r_{y,y'}(q^{-1/2}).
	\end{align}
    See \cite[\S A]{KL79}.
    \\
    \par 
	Decomposing $X^z_x$ using the stratification (\ref{Richardson_strat}), we may describe the singularity type of a point $u\in (X^y_{y'})^\circ\subset X^z_x$ using the local description of Knutson, Woo, and Yong \cite[\S1]{KWY13}.
	\begin{prop}
\label{Richardson_link}
	The singularities occurring in Richardson varieties are products of singularities in Schubert varieties, up to products with affine spaces. 
	Let $u$ be a geometric point in $(X^y_{y'})^\circ\subset X^z_x$, and $(u_1,u_2)$ points in the open cell $(X_{y'})^\circ\times (X^y)^\circ$ of the product $X_{y'}\times X^y$. 
	Then there exists an isomorphism $\gamma$ from a neighbourhood $U$ of $u$ cross an affine space with a neighbourhood $U_1\times U_2$ of $(u_1,u_2)$ in $X_x\times X^z$, such that $\gamma(u,0)=(u_1,u_2)$:
	\begin{align*}
    	U\times \mathbb{A}^{|w_0|}~ &\tilde{\longrightarrow} ~ U_1\times U_2\subset X_x\times X^z\\
    	(u,0)~ &\! \longmapsto \gamma(u,0)=(u_1,u_2).
	\end{align*}
	\end{prop}
\begin{rem}
In \cite{KWY13} the field of definition is assumed to be algebraically closed, but the argument adapts to our case by descent, as the chart of their Lemma 2.2 is defined over $\mathbb{Z}$.
\end{rem}

This proposition allows us to reduce the study of singularities in Richardson varieties to those of Schubert varieties. Schubert varieties admit well understood Bott--Samelson resolutions of their singularities, and in all characteristics, fibre products of these provide resolutions for Richardson varieties \cite[\S A]{KLS14}.

\begin{prop}\label{Prop:res of richardsons}
    Let $\underline{x},\underline{z}$ be reduced expressions for $x$ and $z$, with $X_{\underline{x}}\to X_x$ and $X^{\underline{z}}\to X^z$ the Bott--Samelson resolutions of the associated Schubert and opposite Schubert varieties. Then, as maps to $G/B$, their fibre product is smooth, and yields a resolution of the Richardson variety: \[X^{\underline{z}}_{\underline{x}}:=X_{\underline{x}}\times_{G/B} X^{\underline{z}}\to X^z_x.
    \]
\end{prop}

\subsection{The main geometric argument}
\label{Subsec: geometric argument}
In this section we prove the polynomial expression of Proposition \ref{prop:the polynomial formula} computes the graded rank of the cohomology of a certain $\ell$-adic sheaf $\mathscr{E}({X}^z_x)$ on the Richardson variety. 
The sheaf $\mathscr{E}({X}^z_x)$ is a motivic variant of the geometric extension as introduced in \cite{HW23,McN18}, and its stalks measure the singularities of ${X}^z_x$.
Decomposing along the stratification of (\ref{Richardson_strat}), the stalks of $\mathscr{E}({X}^z_x)$ decompose as a product of the stalks in the associated Schubert varieties. 
These are encoded by the polynomials ${}^\ell {h}_{x,y}$, and the purity of the relevant cohomology groups allows one to prove Theorem \ref{thm:geometric ranks} using the Grothendieck--Lefschetz trace formula. In this section we only use formal properties of the sheaves $\mathscr{E}({X}^z_x)$, deferring the construction and verification of these properties to Section \ref{subsection:geometric extensions on rich}.
\\
\par
In what follows we work with Richardson varieties $X$ over a finite field $\mathbb{F}_p$ for $p\neq \ell$, on the flag variety $G/B$ rather than the Langlands dual $\check{G}/\check{B}$. We work within a\footnote{One may use the $\mathbb{Z}/\ell^n\mathbb{Z}$-limit formalism following \cite{E07}, or directly as constructible sheaves for the pro-\'{e}tale topology of \cite{BS13}.} six functor formalism $X\mapsto D^b_c(X,\mathbb{Z}_\ell)$ of mixed $\mathbb{Z}_\ell$-adic sheaves on such varieties, and all functors between sheaf categories are derived. We require the integral structure for splitting purposes in the definition of $\mathscr{E}^z_x$, but all ranks and (super)traces are be taken after tensoring with $\mathbb{Q}_\ell$. We denote the Tate twist by $\{1\}$. In this section we work with the normalised parity sheaf $\mathscr{E}(X_x)$ and geometric extension $\mathscr{E}(X_x^z)$, such that they extend the constant sheaf in degree zero on the smooth locus. These differ by a shift in the dimension of the associated variety, explicitly: \begin{align*}
    \mathscr{E}_x=\mathscr{E}(X_x)(|x|), &&
     \mathscr{E}^z_x =\mathscr{E}(X^z_x)(|x|-|z|).
\end{align*} This is convenient for point counting, but we warn the reader that $\mathscr{E}(X_x)$ and $\mathscr{E}(X^z_x)$ are only self dual up to a shift and Tate twist.

\begin{prop}\label{Prop:geo properties}
On the Richardson variety $X^z_x$ there exists a  mixed $\ell$-adic sheaf $\mathscr{E}(X^z_x)$, a motivic variant of the geometric extension, which has the following properties:
\begin{enumerate}
    \item The group $\mathbb{H}^0 (\mathscr{E}(X^z_x))$ is nonzero, and the cohomology $\mathbb{H}^\bullet (\mathscr{E}(X^z_x))$ is pure; the eigenvalues of the Frobenius on $\mathbb{H}^i(\mathscr{E}(X^z_x))$ have absolute value $q^{i/2}$.
    \item The sheaf $\mathscr{E}(X^z_x)$ is pointwise pure and the Frobenius eigenvalues on the stalks are all (integer) powers of $q$.
    \item With respect to the natural stratification of Richardson varieties \eqref{Richardson_strat}, the stalk at a point $u$ in $(X^{y}_{y'})^\circ$ is a product of the stalks of the (normalised) parity sheaves on the associated Schubert varieties:\[i_u^*\mathscr{E}(X^z_x)\cong i_{y'}^*\mathscr{E}(X_x)\otimes i_{y}^*\mathscr{E}(X^z).\]
    
\end{enumerate}
\end{prop}
With these properties, we may prove Theorem \ref{thm:geometric ranks} by computing the trace of the Frobenius on the cohomology of $\mathscr{E}(X^z_x)$ in two ways.
\begin{thm}\label{Thm:geo argument}
The graded rank of $\mathbb{H}^\bullet(\mathscr{E}(X^z_x))$ is given by the following formula: 

\[\grrk \mathbb{H}^\bullet(\mathscr{E}(X^z_x))=v^{|x|-|z|}\sum_{z \leq y \leq y' \leq x} {}^{\ell} h_{w_0 y, w_0 z}(v) ~ r_{y,y'}(v) ~{}^{\ell} h_{y',x}(v).\]
This is a polynomial in $v^2$, and the Frobenius acts on $\mathbb{H}^{2i}(\mathscr{E}(X^z_x))$ as multiplication by $q^i$.
\end{thm}

\begin{proof}
Let us compute the trace of the $q$-th power Frobenius on $\mathbb{H}^\bullet(\mathscr{E}(X^z_x))$ using the  Grothendieck--Lefschetz trace formula \cite{De77}. This trace is given as the sum:
\[\text{Tr}_{\text{Frob}}(\mathbb{H}^\bullet(\mathscr{E}(X^z_x)))=\sum_{u\in X^z_x(\mathbb{F}_q)}\text{Tr}_{\text{Frob}}(i_u^*\mathscr{E}(X^z_x)).\]
Decomposing the $\mathbb{F}_q$-points of $X^z_x$ along our stratification, our latter two properties give that these traces are constant on the subset $(X^{y}_{y'})^\circ(\mathbb{F}_q)$, with value equal to the trace of Frobenius on $i_y^*\mathscr{E}(X^z)\otimes i_{y'}^*\mathscr{E}(X_x)$:
\begin{align*}
        \sum_{u\in X^z_x(\mathbb{F}_q)}\text{Tr}_{\text{Frob}}(i_u^*\mathscr{E}(X^z_x))=&
    \sum_{z\leq y\leq y'\leq x}|(X^{y}_{y'})^\circ(\mathbb{F}_q)|~\text{Tr}_{\text{Frob}}(i_{y}^*\mathscr{E}(X^z) \otimes i_{y'}^*\mathscr{E}(X_x))\\
    =&\sum_{z\leq y\leq y'\leq x}|(X^{y}_{y'})^\circ(\mathbb{F}_q)|~\text{Tr}_{\text{Frob}}(i_{y}^*\mathscr{E}(X^z))~  \text{Tr}_{\text{Frob}}(i_{y'}^*\mathscr{E}(X_x)).
    \end{align*}
    Recall that the point count of the open Richardson variety $(X_{y'}^{y})^{\circ}$ gives the $r_{y,y'}$ polynomial
    \[|(X^y_{y'})^\circ(\mathbb{F}_q)|=R_{y,y'}(q)=q^{\frac{|y'|-|y|}{2}}r_{y,y'}(q^{-1/2}).\]
    The trace of the Frobenius on these parity stalks gives their Poincar\'{e} polynomial \cite[\S 1.3.5]{JMW14}, which by definition are the $\ell$-Kazhdan--Lusztig polynomials:
    \begin{align*}
    \text{Tr}_{\text{Frob}}(i_{y'}^*\mathscr{E}(X_x))&=\sum_{i \in \Z}  \rank_{\mathbb{Q}_\ell} H^i (i_{y'}^*\mathscr{E}(X_{x})) q^{\frac{i}{2}}=q^{\frac{|x|-|y'|}{2}}~ {}^{\ell} h_{y',x}(q^{-1/2}) \\
    \text{Tr}_{\text{Frob}}(i_{y}^*\mathscr{E}(X^z))&=\text{Tr}_{\text{Frob}}(i_{w_0y}^*\mathscr{E}(X_{w_0 z}))=q^{\frac{|y|-|z|}{2}}~ {}^{\ell} h_{w_0y,w_0z}(q^{-1/2}). 
    \end{align*}
    This then gives the identification of traces:
    \begin{align*}
        \text{Tr}_{\text{Frob}}(\mathbb{H}^\bullet(\mathscr{E}(X^z_x)))=&\sum_{z \leq y \leq y' \leq x} q^{\frac{|y|-|z|}{2}} {}^{\ell} h_{w_0y,w_0z}(q^{-1/2})~ q^{\frac{|y'|-|y|}{2}}r_{y,y'}(q^{-1/2})~ q^{\frac{|x|-|y'|}{2}} {}^{\ell} h_{y',x}(q^{-1/2})\\
        = &~ q^{\frac{|x|-|z|}{2}}\sum_{z \leq y \leq y' \leq x} {}^{\ell} h_{w_0y,w_0z}(q^{-1/2})~r_{y,y'}(q^{-1/2})~{}^{\ell} h_{y',x}(q^{-1/2})\\
        = &~ q^{\frac{|x|-|z|}{2}}\sum_{z \leq y \leq y' \leq x} {}^{\ell} h_{w_0y,w_0z}(q^{1/2})~r_{y,y'}(q^{1/2})~{}^{\ell} h_{y',x}(q^{1/2}).
    \end{align*}
    The final symmetry follows from Corollary \ref{cor:numerical duality} and Proposition \ref{prop:the polynomial formula}. We claim this expression is a Laurent polynomial in $q$ rather than $q^{1/2}$. To see this, note that ${}^{\ell} h_{y,x}(v)$ have coefficients concentrated in a single parity, which implies the claim in view of Proposition \ref{prop:the polynomial formula}. The total parity of this expression is then seen to be even by the nontriviality of $\mathbb{H}^0(\mathscr{E}(X^z_x))$.
    The purity of the Frobenius action then implies by Lemma \ref{lem:lin_alg} that this trace is the Poincar\'{e} polynomial of $\mathbb{H}^\bullet(\mathscr{E}(X^z_x))$ with variable $q^{1/2}$, giving the theorem.
   \end{proof}

In the proof of Theorem \ref{Thm:geo argument} we required the following lemma. 

\begin{lem}\label{lem:lin_alg}
Given a semisimple linear map $F$ of degree zero on a graded $\mathbb{Q}_\ell$-vector space $V\cong \bigoplus_i V_i$, such that:
\begin{itemize}
\item The eigenvalues\footnote{Interpreted in $\overline{\mathbb{Q}}_\ell \cong \mathbb{C}$ in the usual manner.} of $F$ on $V_i$ all have absolute value $q^{i/2}$.
\item For all $k\geq 0$, the (super)trace of $F^k$ on $V$ is a Laurent polynomial $Q_V$ in $q^k$:\[\Tr_{F^k}(V)=Q_V(q^k)=\sum_i a_i q^{ki}.\]
\end{itemize}
Then $V_{2i+1}=0$, and on the $V_{2i}$ component, $F$ acts by multiplication by $q^i$, and we have \[\Tr_{F^k}(V)=Q_V(q^k)=P_V(q^{k/2})\] for the Poincar\'{e} polynomial $P_V$ of $V$ with variable $q^{1/2}$.
\end{lem}
\begin{proof}
    For any linear operator $F$ on a finite dimensional vector space we have an equality of formal power series \[\sum_{k\geq 0} \Tr_{F^k}(V)t^k=\sum_{j}\frac{n_j}{1-\lambda_j t}\] where $\lambda_j$ are the generalised eigenvalues of $F$ with multiplicities $n_j$. Comparing this with \[\sum_{k\geq 0}P(q^k)t^k = \sum_i \frac{a_i}{1-q^i t}\] gives the conclusion by comparing roots and multiplicities.
\end{proof}

\subsection{Geometric extensions on Richardson varieties}
\label{subsection:geometric extensions on rich}

In this section we define the sheaves $\mathscr{E}(X^z_x)$ and prove Proposition \ref{Prop:geo properties}. The sheaves $\mathscr{E}(X^z_x)$ are a (numerical) motivic variant of the geometric extension introduced in \cite{HW23,McN18}. These sheaves exist on an arbitrary variety admitting a resolution of singularities, and may be described as follows.

\begin{prop}\label{prop:char of geo exts}
Let $Y$ be an irreducible $
\mathbb{F}_q$-variety which admits a resolution of singularities. There exists a canonical mixed $\ell$-adic sheaf $\mathscr{E}(Y)$ in $D^b_c(Y,\mathbb{Z}_\ell)$, the geometric extension, characterised by the following properties:\begin{enumerate}
    \item The sheaf $\mathscr{E}(Y)$ is an extension of the constant sheaf over the smooth locus $U$ of $Y$:
    \[\mathscr{E}(Y)|_U\cong \mathbf{1}_U.\]
    \item For \textbf{any} resolution of singularities $f:X\to Y$, the geometric extension $\mathscr{E}(Y)$ is a summand of $f_*\mathbf{1}_X$.
    \item For a resolution $f:X\to Y$, the idempotent endomorphism of $f_*\mathbf{1}_X$ cutting out $\mathscr{E}(Y)$ is in the image of the $\mathbb{Z}_\ell$-cycle class map: \[\CH_{d_Y}(X\times_Y X)\otimes \mathbb{Z}_\ell\to \hH^{\text{BM}}_{2d_Y}(X\times_Y X,\mathbb{Z}_\ell\{-d_Y\})\cong \Hom^0(f_*\mathbf{1}_X,f_*\mathbf{1}_X)\] and this idempotent does not split as a sum of smaller idempotents in this image.
\end{enumerate}
\end{prop}

\begin{proof}
    The proof of this is a slight modification of that in \cite{HW23,McN18}; we refer to them for more details. 
    Consider the category  $\text{Res}_{/Y}$ of resolutions of $Y$, with objects $f:X\to Y$ with $X$ smooth, $f$ a proper birational map, with morphisms given by the numerical Chow groups
    \[\Hom((X,f),(Z,g)):=\CH_{d_Y}(X\times_Y Z)_{num,\mathbb{Z}_\ell}\subset \text{H}_{2d_Y}^{\text{BM}}(X\times_Y Z,\mathbb{Z}_\ell\{-d_Y\})\] 
    under convolution product, where $\CH_{d_Y}(X\times_Y Z)_{num,\mathbb{Z}_\ell}$ is the image of the $\mathbb{Z}_\ell$-cycle class map to $\ell$-adic Borel--Moore homology (see \cite[\S6]{Mi80}, \cite[\S4]{De77}). 
    These homsets are finitely generated $\mathbb{Z}_\ell$-modules, so we may apply the Krull--Schmidt techniques of \cite[\S4]{HW23}. 
    In particular, for any resolution $f:X\to Y$, there is a minimal idempotent $e$ of $(X,f)$ such that $\text{Im}(f_*\mathbf{1}_X\xrightarrow{e} f_*\mathbf{1}_X)$ extends the constant sheaf over $U$. 
    For any two resolutions, the fundamental class of the closure of the diagonal $\Delta(U)$ induces a morphism between these idempotents, and this map is an isomorphism as it is defined symmetrically, and is an isomorphism over $U$. 
    The image of this idempotent is therefore independent of the resolution chosen.
\end{proof}
We obtain purity of the cohomology of $\mathscr{E}(Y)$ at once from this definition.
\begin{prop}
For any proper $Y$, the cohomology $\mathbb{H}^\bullet(\mathscr{E}(Y))\otimes_{\mathbb{Z}_\ell} \mathbb{Q}_\ell$ is pure.
\end{prop}

\begin{proof}
This cohomology is a direct summand of the cohomology $\hH^\bullet(X,\mathbb{Q}_\ell)$ for $X$ smooth and proper, which is pure by the foundational results of Deligne \cite{De80}.
\end{proof}

The construction of $\mathscr{E}(Y)$ shows the group $\mathbb{H}^0(\mathscr{E}(Y))$ is nonzero, as the canonical map \[\mathbf{1}_Y\to f_*\mathbf{1}_X\to \mathscr{E}(Y)\] is nonzero, being an isomorphism over $U$.
\begin{rem}
The geometric extension $\mathscr{E}(Y)$ is isomorphic to its Verdier dual up to a shift and twist: \[\mathbb{D}\mathscr{E}(Y)\cong \mathscr{E}(Y)(2d_Y)\{d_Y\}.\] This symmetry is a geometric incarnation of the symmetry of Corollary \ref{cor:numerical duality}.
\end{rem}

\begin{prop}\label{prop:locality of geo ext}
    For $V$ open in $Y$, the restriction of $\mathscr{E}(Y)$ to $V$ is given by $\mathscr{E}(V)$.
\end{prop}
\begin{proof}
    By the minimality of the construction,  $\mathscr{E}(V)$ occurs as a summand of $\mathscr{E}(Y)|_V$, so it suffices to show this sheaf does not split upon restriction to $V$. For a resolution $f:X\to Y$, with idempotent $e:f_*\mathbf{1}_X\to f_*\mathbf{1}_X$ cutting out $\mathscr{E}(Y)$, we need to check that $e$ cannot be decomposed over $V$ into \[e|_V=e_1+e_2\] where $e_i$ are in the image of a $\mathbb{Z}_\ell$-cycle. For this, we have the commutative diagram \[\begin{tikzcd} \CH_{d_Y}(X\times_Y X)\otimes \mathbb{Z}_\ell\arrow[rr,bend left]\arrow[d,two heads]\arrow[r,two heads]&\End_{\text{Res}{_{/Y}}}((X,f))\arrow[r,hook]\arrow[d,"\pi",two heads]&\Hom^0(f_*\mathbf{1}_X,f_*\mathbf{1}_X)\arrow[d]\\
    \CH_{d_Y}(f^{-1}(V)\times_Vf^{-1}(V))\otimes \mathbb{Z}_\ell\arrow[rr,bend right]\arrow[r,two heads]&\End_{\text{Res}{_{/V}}}((f^{-1}(V),f))\arrow[r,hook]&\Hom^0(f_*\mathbf{1}_{f^{-1}(V)},f_*\mathbf{1}_{f^{-1}(V)})
    \end{tikzcd}\] This left map is surjective by the localisation sequence in Chow groups, implying surjectivity of $\pi$. If the idempotent $e$ splits under $\pi$, so $\pi(e)=e_1+e_2$ with $\pi(e)\cdot e_1=e_1$, we may find $\tilde{e}_1\in \End_{\text{Res}{_{/Y}}}((X,f))$ with $\pi(\tilde{e}_1)=e_1$. The algebra $e\cdot \End_{\text{Res}{_{/Y}}}((X,f))\cdot e$ is local by definition, so the element $e\cdot \tilde{e}_1\cdot e$ is either a unit or contained in the unique maximal ideal. By symmetry we may assume $e\cdot \tilde{e}_1\cdot e$ is a unit, so \[e_1=\pi(e)\cdot \pi(\tilde{e}_1)\cdot \pi(e)=\pi(e\cdot \tilde{e}_1\cdot e)\] is a unit, giving $e_1=1$, as desired.
\end{proof}

\begin{rem}
The locality property of Proposition \ref{prop:locality of geo ext} necessitated the motivic variant of the geometric extension, as opposed to that of \cite{HW23, McN18}. We do not know whether this property holds for the minimal dense indecomposable summand of $f_*\mathbf{1}_X$ in full generality.
\end{rem}

\begin{prop}\label{prop:geo is parity}
For a Schubert variety $X_w$, the geometric extension $\mathscr{E}(X_w)$ is the indecomposable parity sheaf on $X_w$, shifted to be the constant sheaf in degree $0$ over the smooth locus.
\end{prop}

\begin{proof}
The Bott--Samelson resolution $X_{\underline{w}}\to X_w$, for a reduced word $\underline{w}$ of $w$, has the property that the fibre product $X_{\underline{w}}\times_{X_w} X_{\underline{w}}$ admits an affine paving. As such, the $\mathbb{Z}_\ell$-cycle class map is an isomorphism, so we deduce that the geometric extension is indecomposable in the ambient constructible derived category. The geometric extension is therefore a dense indecomposable parity sheaf on $X_w$, so is the parity sheaf $\mathscr{E}_w$ up to the prescribed shift.
\end{proof}

\begin{cor}\label{cor:kunneth}
On Schubert varieties $X_w$ and $X_u$, we have an isomorphism: \[\mathscr{E}(X_w)\boxtimes \mathscr{E}(X_u)\cong \mathscr{E}(X_w\times X_u).\]
\end{cor}

\begin{proof}
This follows from the corresponding statement for parity sheaves, and may be seen directly via the isomorphism 
\begin{align*}
		\CH_\bullet((X_{\underline{w}}\times_{X_w} X_{\underline{w}})\times (X_{\underline{u}}\times_{X_u} X_{\underline{u}}))
		\cong \CH_\bullet(X_{\underline{w}}\times_{X_w} X_{\underline{w}})\otimes \CH_\bullet(X_{\underline{u}}\times_{X_u} X_{\underline{u}}),
\end{align*}
which holds as the relevant varieties admit affine pavings.
\end{proof}

\begin{prop}\label{prop:homotopy invariance}
The geometric extension on the product $Y\times \mathbb{A}^n$ is given by the external product of $\mathscr{E}(Y)$ and the constant sheaf:

\[\mathscr{E}(Y\times \mathbb{A}^n)\cong \mathscr{E}(Y)\boxtimes \mathbf{1}_{\mathbb{A}^n}.\]
\end{prop}

\begin{proof}
For a resolution $f:X\to Y$, taking products with $\mathbb{A}^n$ we obtain a commutative square:
\[\begin{tikzcd} \CH_{d_Y}(X\times_Y X)\arrow[r]\arrow[d,"\sim"]&\Hom^0(f_*\mathbf{1}_X,f_*\mathbf{1}_X)\arrow[d,"\sim"]\\
    \CH_{d_Y+n}((X\times \mathbb{A}^n)\times_{Y\times\mathbb{A}^n}(X\times \mathbb{A}^n))\arrow[r]&\Hom^0((f\times \id)_*\mathbf{1}_{X\times \mathbb{A}^n},(f\times \id)_*\mathbf{1}_{X\times \mathbb{A}^n})
    \end{tikzcd}\]
These vertical maps are isomorphisms by homotopy invariance of Chow groups and sheaves respectively. As such, the induced map on images is an isomorphism:
\[\End_{\text{Res}_{/Y}}(X,f)\xrightarrow{\sim} \End_{\text{Res}_{/{Y\times \mathbb{A}^n}}}(X\times \mathbb{A}^n,f\times \text{id}).\]
That this algebra isomorphism is induced by crossing with $\mathbb{A}^n$ then gives the claim.
\end{proof}
    We now prove the desired stalk description, using the local description of Proposition \ref{Richardson_link}.

\begin{prop}\label{prop:stalks agree}
The geometric extension on the Richardson variety is pointwise pure, and we can identify the stalk of $\mathscr{E}(X^z_x)$ at a point $u$ in $(X^y_{y'})^\circ\subset X^z_x$ with: \[i_u^*\mathscr{E}(X^z_x)\cong i_{y'}^*\mathscr{E}(X_x)\otimes i_y^*\mathscr{E}(X^z).\]
\end{prop}

\begin{proof}
Our previous lemmas give that these stalks are $\mathbb{A}^n$ homotopy invariant, Zariski local, and are locally described by the stalks of parity sheaves on Schubert varieties. We therefore have the following isomorphisms:
\begin{align*}
i^*_u\mathscr{E}({X^z_x})&\cong i^*_{u,0}(\mathscr{E}(X^z_x) \boxtimes \mathbf{1}_{\mathbb{A}^{|w_0|}})\\
&\cong i^*_{u,0}(\mathscr{E}(X_z^x\times \mathbb{A}^{|w_0|}))&&\text{Proposition \ref{prop:homotopy invariance}}\\
&\cong i^*_{u_1,u_2}(\mathscr{E}(X_x\times X^z))&&\text{Propositions \ref{prop:locality of geo ext} and \ref{Richardson_link}}\\
&\cong i^*_{y',y}\mathscr{E}(X_x\times X^z)&&\text{Constructibility along Schubert stratification}\\
&\cong i^*_{y'}(\mathscr{E}(X_x))\otimes i^*_y(\mathscr{E}(X^z))&&\text{Corollary \ref{cor:kunneth}.}
\end{align*}
\end{proof}

\begin{cor}
The sheaf $\mathscr{E}(X^z_x)$ is pointwise pure, and the eigenvalues of Frobenius on each stalk are powers of $q$.
\end{cor}

\begin{proof}
This local statement holds since it holds for parity sheaves \cite[1.3.5]{JMW14}. Explicitly, the stalks of the parity sheaf are direct summands of the cohomology of the fibres of any Bott--Samelson resolution of $X_x$. These fibres are complete and admit affine pavings, so the Frobenius on their cohomology has the desired properties by reduction to the case of $\mathbb{A}^n$.
\end{proof}

This verifies the properties stated in Proposition \ref{Prop:geo properties}.
The argument in Theorem \ref{Thm:geo argument} then exhibits the ranks of cohomology as the polynomial expression for the tilting multiplicities of Proposition \ref{prop:the polynomial formula}. From this we deduce Theorem \ref{thm:main thm}.

\subsection{A tensor product description of the geometric extension}
\label{Ssec: Tensor product}

We conclude this section with a more explicit description of this geometric extension, as a tensor product of parity sheaves for the Bruhat and opposite Bruhat stratification.
\begin{thmp}{\ref{Thm: Tensor product}}
\emph{
The geometric extension on the Richardson variety is the tensor product \[\mathscr{E}^z_x\cong \mathscr{E}_x\otimes \mathscr{E}^z\] where $\mathscr{E}_x$ and $\mathscr{E}^z$ denote dense indecomposable parity sheaves on the Schubert and opposite Schubert varieties respectively.}
\end{thmp}

\begin{proof}
For notational simplicity we write $f_* \mathbf{1}$ for the pushforward of the constant sheaf on the domain of $f$.
By Proposition \ref{Prop:res of richardsons}, the fibre product (in the flag variety) of the Bott--Samelson resolutions $\pi_{\underline{x}}: X_{\underline{x}}\to X_x$ and $\pi^{\underline{z}}:X^{\underline{z}}\to X^z$ is smooth, and yields a resolution \[\pi_{\underline{x}}^{\underline{z}}:X^{\underline{z}}_{\underline{x}}\to X^z_x\] of the Richardson variety. 
The (normalised) parity sheaves $\mathscr{E}(X_x)$ and $\mathscr{E}(X^z)$ are cut out of $\pi_{\underline{x}*}\mathbf{1}$ and $(\pi^{\underline{z}})_{\ast}\mathbf{1}$ by certain idempotent endomorphisms $e_x$ and $e^z$ respectively, and these idempotents lie in the image of the cycle class map by Proposition \ref{prop:geo is parity}. The external product of these idempotents can be pulled back to the diagonal to yield an idempotent endomorphism $e_x^z$ of $(\pi_{\underline{x}}^{\underline{z}})_\ast\mathbf{1}$ by proper base change:
\[\begin{tikzcd}
    \Delta^*(\pi_{\underline{x}\ast}\mathbf{1}\boxtimes (\pi^{\underline{z}})_{*}\mathbf{1})\arrow[r,"\sim"]\arrow[d,"\Delta^*(e_x\boxtimes e^z)"]&\Delta^*(\pi_{\underline{x}}\times \pi^{\underline{z}})_*\mathbf{1}\arrow[d]\arrow[r,"\sim"]&(\pi_{\underline{x}}^{\underline{z}})_\ast\mathbf{1}\arrow[d,"e^z_x"]\\
    \Delta^*(\pi_{\underline{x}\ast}\mathbf{1}\boxtimes (\pi^{\underline{z}})_*\mathbf{1})\arrow[r,"\sim"]&\Delta^*(\pi_{\underline{x}}\times \pi^{\underline{z}})_*\mathbf{1}\arrow[r,"\sim"]&(\pi_{\underline{x}}^{\underline{z}})_\ast\mathbf{1}.
\end{tikzcd}\]
For the moment, let us assume that this pulled back idempotent $e^z_x$ is in the image of the cycle class map. From the characterisation of $\mathscr{E}(X^z_x)$ in Proposition \ref{prop:char of geo exts}, we conclude that $\mathscr{E}(X^z_x)$ is a summand of $\mathscr{E}(X_x)\otimes \mathscr{E}(X^z)$. As we know their stalks are equal by Proposition \ref{prop:stalks agree}, this split inclusion is an isomorphism, and renormalising then yields the result.
\\
\par
The remainder of this proof is showing that this idempotent is in the image of the cycle class map. The reason for this is that we may recognise this pullback as the refined intersection with the fundamental class of $\mathcal{F} = G/B$. That is, the following diagram commutes, where the horizontal maps are the realisation maps from cycles to (degree zero) endomorphisms of sheaves:
\[\begin{tikzcd}
    \CH_{|x|+|w_0 z|}((X_{\underline{x}}\times_{\mathcal{F}}X_{\underline{x}})\times (X^{\underline{z}}\times_{\mathcal{F}}X^{\underline{z}}))\arrow[r]\arrow[dd,"\cap \Delta \mathcal{F}"]&\End^0(\pi_{\underline{x}\ast}\mathbf{1}\boxtimes (\pi^{\underline{z}})_*\mathbf{1})\arrow[d,"\Delta^*"]\\
    &\End^0(\Delta^*(\pi_{\underline{x}\ast}\mathbf{1}\boxtimes (\pi^{\underline{z}})_*\mathbf{1}))\arrow[d,"\sim"]\\
    \CH_{|x|-|z|}(X_{\underline{x}}^{\underline{z}}\times_{\mathcal{F}}X_{\underline{x}}^{\underline{z}})\arrow[r]&\End^0((\pi_{\underline{x}}^{\underline{z}})_*\mathbf{1}).
\end{tikzcd}\]
\par 
As we are not aware of this exact compatibility in the literature, we include a proof.
To check that this diagram\footnote{Suppressing Tate twists for clarity} commutes, we first need that taking refined intersection commutes with cycle class maps:
\[\begin{tikzcd}
    \CH_{|x|+|w_0z|+\bullet}((X_{\underline{x}}\times_{\mathcal{F}}X_{\underline{x}})\times (X^{\underline{z}}\times_{\mathcal{F}}X^{\underline{z}}))\arrow[r]\arrow[d,"\cap \Delta \mathcal{F}"]&\hH^{\text{BM}}_{2(|x|+|w_0z|+\bullet)}((X_{\underline{x}}\times_{\mathcal{F}}X_{\underline{x}})\times (X^{\underline{z}}\times_{\mathcal{F}}X^{\underline{z}}))\arrow[d,"\cap \Delta \mathcal{F}"]\\
    \CH_{|x|-|z|+\bullet}(X_{\underline{x}}^{\underline{z}}\times_{\mathcal{F}}X_{\underline{x}}^{\underline{z}})\arrow[r]&\hH^{\text{BM}}_{2(|x|-|z|+\bullet)}(X_{\underline{x}}^{\underline{z}}\times_{\mathcal{F}}X_{\underline{x}}^{\underline{z}}).
\end{tikzcd}\]
This follows from the construction of the refined intersection map, using the Gysin maps as in \cite{F84}, which are compatible in the \'{e}tale realisation, see \cite[\S23]{Mi80}. 
We may view these Borel-Moore homology groups as the cohomology of the respective dualising sheaves. With respect to a pullback diagram
\[\begin{tikzcd}
W\cap Z\arrow[r,"i'"]\arrow[d,"f'"]&W\arrow[d,"f"]\\
Z\arrow[r,"i"]&X
\end{tikzcd}\] where $Z,X$ are smooth, and $i$  is a closed immersion of codimension $c$, we may describe the intersection product six functorially as the following composition:
\[\begin{tikzcd}
    \hH_{-\bullet}^{\text{BM}}(W)\arrow[rr,"\cap Z"]\arrow[d,"\simeq"]&& \hH_{-\bullet+2c}^{\text{BM}}(W\cap Z)\\
    \hom^{\bullet}(\mathbf{1}_W,\omega_W)\arrow[d]&& \hom^\bullet(\mathbf{1}_{W\cap Z},\omega_{W\cap Z}(2c)\{c\})\arrow[u,"\simeq"]\\
    \hom^{\bullet}((i')^*\mathbf{1}_W,(i')^*f^!\omega_X)\arrow[r]& \hom^\bullet(\mathbf{1}_{W\cap Z},(f')^!i^*\omega_X)\arrow[r,"\simeq"]& \hom^\bullet(\mathbf{1}_{W\cap Z},f^!i^*\mathbf{1}_X(2d_X)\{d_X\})\arrow[u,"\simeq"].
\end{tikzcd}\]
This allows us to reduce to a formal compatibility of functors, evaluated at constant and dualising sheaves. With respect to a pullback cube 
\[\begin{tikzcd}
Z\times_X W_1\times_X W_2 \arrow[rr,"\hat{f}"]\arrow[dr,"\hat{k}"]\arrow[dd,"i_{12}"]&&Z\times_X W_2\arrow[dd,"i_2",near end]\arrow[dr,"k'"]\\
&Z\cap W_1\arrow[rr,"f'",near end]\arrow[dd,"i_1",near end]&&Z\arrow[dd,"i"]\\
W_1\times_X W_2\arrow[rr,"\tilde{f}",near end]\arrow[dr,"\tilde{k}"]&&W_2\arrow[dr,"k"]\\
&W_1\arrow[rr,"f"]&&X
\end{tikzcd}\]
the compatibility is then the commutativity of the following six functorial natural transformations: 
\begin{equation}\label{Eqn:commuting diagram}
\begin{tikzcd}
\hom((\tilde{k})^*,(\tilde{f})^!)\arrow[r]\arrow[d]&\hom(f_!,k_*)\arrow[d]\\
\hom((i_{12})^*(\tilde{k})^*,(i_{12})^*(\tilde{f})^!)\arrow[d]&\hom(i^*f_!,i^*k_*)\arrow[d]\\
\hom((\hat{k})^*(i_1)^*,(\hat{f})^!(i_2)^*)\arrow[r]&\hom((f')_!(i_1)^*,(k')_*(i_2)^*).
\end{tikzcd}
\end{equation}
In our case we take $X=\mathcal{F}\times \mathcal{F}$, $Z=\Delta(\mathcal{F})$, $W_1=W_2=X_{\underline{x}}\times X^{\underline{z}}$. 
A proof of commutativity of (\ref{Eqn:commuting diagram}) may be recovered from the following diagram, where the commutativity of the individual regions are standard, and all maps are the canonical ones:
\[\begin{tikzcd}
(\tilde{k})_*(\tilde{f})^! \arrow[rd] \arrow[dd] \arrow[rr] &                                  & f^!k_* \arrow[d] \arrow[r] & f^!i_*i^*k_* \arrow[d] \\
                                        & (i_1)_*(i_1)^*(\tilde{k})_*(\tilde{f})^! \arrow[r] \arrow[d] & (i_1)_*(i_1)^*f^!k_* \arrow[r]     & (i_1)_*(f')^!i^*k_* \arrow[d] \\
(\tilde{k})_*(i_{12})_*(i_{12}^*)(\tilde{f})^! \arrow[r]                  & (i_1)_*(\hat{k})_*(i_{12})^*(\tilde{f})^! \arrow[r]           & (i_1)_*(\hat{k})_*(\hat{f})^!(i_2)^*) \arrow[r]     & (i_1)_*(f')^!(k')_*(i_2)^*     .     
\end{tikzcd}\]
One may also conclude this diagram (\ref{Eqn:commuting diagram}) commutes by the general commutativity Theorem 3.1.1 of \cite[\S3]{Hon25}.

\end{proof}




\Addresses


{\footnotesize
\begin{thebibliography}{XXXX.}


\bibitem[AMRW]{AMRW19}
P. Achar, S. Makisumi, S. Riche and G. Williamson,
\emph{Koszul duality for Kac-Moody groups and characters of tilting modules},
J. Amer. Math. Soc. \textbf{32} (2019), no. 1, 261-310.

\bibitem[AR1]{AR16a}
P. Achar and S. Riche,
\emph{Modular perverse sheaves on flag varieties I: tilting and parity sheaves},
Ann. Sci. \'{E}c. Norm. Sup\'{e}r. (4) \textbf{49} (2016), no. 2, 325-370.

\bibitem[AR2]{AR16}
P. Achar and S. Riche, 
\emph{Modular perverse sheaves on flag varieties, II: Koszul duality and formality},
Duke Math. J.  \textbf{165} (2016), no. 1, 161-215.

\bibitem[Bai]{Bai}
J. Baine, 
\emph{On the coefficients in the Jones-Wenzl idempotent}, 
Jun. 2024, arXiv:2406.06333.

\bibitem[BS]{BS13}
B. Bhatt, P. Scholze,
\emph{The Pro-étale topology for schemes},
2013, Asterisque 369 (2015), 99-201.

\bibitem[BGS]{BGS96}
A. Beilinson, V. Ginzburg and W. Soergel,
\emph{Koszul duality patterns in representation theory},
J. Amer. Math. Soc. \textbf{9} (1996), no.2, 473-527.

\bibitem[Del1]{De77}
P. Deligne, 
\emph{Cohomologie \'{e}tale},
S\'{e}minaire de G\'{e}om\'{e}trie Alg\'{e}brique du Bois-Marie, avec la collaboration de J. F. Boutot, A. Grothendieck, L. Illusie et
J. L. Verdier, Lect. Notes in Math. 569, Springer-Verlag, Berlin-New York,
1977.

\bibitem[Del2]{De80}
P. Deligne, 
\emph{La conjecture de Weil, II},
IHES. Publ. Math. 43 (1980), 157-252.

\bibitem[Deo]{De85}
V. Deodhar,
\emph{On some geometric aspects of Bruhat orderings. I. A finer decomposition of Bruhat cells}, Invent. Math. \textbf{79} 1985, 499-511.

\bibitem[DL]{DL23}
M. Dyer and G. Lusztig,
\emph{A study of intersections of Schubert varieties},
Jul. 2023, arXiv:2307.04646. 

\bibitem[Eke]{E07}
T. Ekedahl,
\emph{On The Adic Formalism},
Jun. 2007, The Grothendieck Festschrift. Modern Birkhäuser Classics. Birkhäuser, Boston, MA.

\bibitem[Erd]{Erd94}
K. Erdmann,
\emph{Symmetric groups and quasi-hereditary algebras},
Finite-dimensional algebras and related topics (Ottawa, ON,
1992). Vol. 424. NATO Adv. Sci. Inst. Ser. C Math. Phys. Sci.
 1994, pp. 123-161.
 
\bibitem[Ful]{F84}
W. Fulton,
\emph{Intersection Theory}, second edition, (1998)
Ergebnisse der Mathematik und ihrer Grenzgebiete. 3. Folge. A Series of Modern Surveys in Mathematics, 2, Springer, Berlin; MR1644323.

\bibitem[GK]{GK92}
S. Gelfand and D. Kazhdan, 
\emph{Examples of tensor categories},
Invent. Math. \textbf{109} (1992), 3, 595-617.

\bibitem[Hon]{Hon25}
C. Hone, 
\emph{Geometric extensions and the six functor formalism},
Ph.D. thesis, University of Sydney, 2025.

\bibitem[HW]{HW23}
C. Hone and G. Williamson, 
\emph{Geometric Extensions}, 
International Mathematics Research Notices (2025), Issue 11.

\bibitem[Jut]{Jut09}
D. Juteau, 
\emph{Decomposition numbers for perverse sheaves}, 
Ann. Inst. Fourier \textbf{59} (2009), 1177-1229.

\bibitem[JMW]{JMW14}
D. Juteau, C. Mautner and G. Williamson,
\emph{Parity sheaves},
J. Amer. Math. Soc. \textbf{27} (2014), no. 4, 1169-1212.

\bibitem[KL]{KL79}
D. Kazhdan and G. Lusztig,
\emph{Representations of Coxeter groups and Hecke algebras},
Invent. Math. \textbf{53} (1979), 2, 165-184.

\bibitem[KLS]{KLS14}
A. Knutson, T. Lam, D. Speyer, 
\emph{Projections of Richardson varieties}
J. Reine Angew. Math (Crelle’s Journal) \textbf{687} (2014), 133-157.

\bibitem[KWY]{KWY13}
A. Knutson, A. Woo and A. Yong,
\emph{Singularities of Richardson varieties},
Math. Res. Lett. \textbf{20} (2013), no. 2, 391-400.

\bibitem[McN]{McN18}
P. McNamara, 
\emph{Non-perverse parity sheaves on the flag variety}, 
Nagoya Math. J. \textbf{249} (2023) 1-10.

\bibitem[Mil]{Mi80}
J. S. Milne,
\emph{Etale Cohomology},
Princeton University Press, 1980, pp. 220-303.

\bibitem[Ric]{Ri92}
R. Richardson,
\emph{Intersections of double cosets in algebraic groups},
Indagationes Mathematicae, 1992, volume \textbf{3}, 69-77.

\bibitem[RV]{RV24}
S. Riche and C. Vay,
\emph{Koszul duality for Coxeter groups}, 
Ann. Repr. Th. 1 (2024) no. 3, pp. 335-374. 

\bibitem[RW]{RW21}
S. Riche and G. Williamson, 
\emph{A simple character formula},
Ann. H. Lebesgue \textbf{4} (2021), 503-535.

\bibitem[Spe]{Sp24}
D. Speyer,
\emph{Richardson varieties, projected Richardson varieties and positroid varieties},
Nov. 2024, arXiv:2303.04831.

\bibitem[Ste]{Ste63}
R. Steinberg,
\emph{Representations of algebraic groups},
Nagoya Math. J. \textbf{22} (1963), 33-56.

\bibitem[Tits]{Tits69}
J. Tits,
\emph{Le probl\`{e}me des mots dans les groupes de Coxeter},
Symposia Mathematica (INDAM, Rome, 1967/68),  \textbf{1}, 175-185, Academic Press, London-New York, 1969.

\end{thebibliography}
}
\end{document}